\documentclass{amsart}

\usepackage{amssymb}
\usepackage{amsthm}
\usepackage{amsmath}
\usepackage{latexsym}
\usepackage{amsbsy}
\usepackage[dvips]{color}
\usepackage{url}

\usepackage{kotex}

\newcommand{\gn}[1]{\langle {#1}\rangle}
\newcommand{\Qd}{/\!\!/}

\newtheorem{thm}{Theorem}[section]
\newtheorem{cor}{Corollary}[thm]
\newtheorem{lem}[thm]{Lemma}
\newtheorem{prop}[thm]{Proposition}
\newtheorem{prob}[thm]{Problem}

\theoremstyle{definition}

\newtheorem*{remark}{Remark}
\newtheorem*{example}{Example}

\title{Association schemes all of whose symmetric fusion schemes are integral}

\author{Mitsugu Hirasaka}
\address{Department of Mathematics, College of Sciences, Pusan National University, 2, Busandaehak-ro 63beon-gil, Geumjung-gu, Busan 46241, Republic of Korea}
\email{hirasaka@pusan.ac.kr}

\author{Kijung Kim}
\address{Department of Mathematics, College of Sciences, Pusan National University, 2, Busandaehak-ro 63beon-gil, Geumjung-gu, Busan 46241, Republic of Korea}
\email{knukkj@pusan.ac.kr}

\author{Semin Oh}
\address{Department of Mathematics, College of Sciences, Pusan National University, 2, Busandaehak-ro 63beon-gil, Geumjung-gu, Busan 46241, Republic of Korea}
\email{semin@pusan.ac.kr}

\begin{document}

\maketitle

  \begin{abstract}
    In this paper we aim to characterize association schemes
    all of whose symmetric fusion schemes have only integral eigenvalues,
    and classify those obtained from a regular action of a finite group
    by taking its orbitals.
  \end{abstract}


\section{Introduction}
In the history of algebraic combinatorics it has been one of the important topics
to consider eigenvalues of the adjacency matrix of a graph.
In \cite{bi} and \cite{bcn} many criterions and conjectures on such problems are suggested and the eigenvalues of well-known distance-regular graphs are explicitly found.
Together with web catalogue \cite{kissme} this gives many association schemes with integral first eigenmatrices (see \cite{bi} and \cite{bcn} for its definition).
As mentioned in \cite[Ex.~2.1]{bi} a transitive permutation group $H$ on a finite set $X$ induces an association scheme $(X,\mathcal{R}_H)$ where $\mathcal{R}_H$ is the set of orbitals of $H$.
If $G$ is a permutation group of $X$ containing $H$, then each element in $\mathcal{R}_G$ is a union of elements of $\mathcal{R}_H$, and the first eigenmatrix of $(X,\mathcal{R}_G)$ is influenced by that of $(X,\mathcal{R}_H)$.
In general, the fusion (and fission) of association relations gives rise to new association schemes from a given association scheme.
In this paper we focus on association schemes whose adjacency matrices have only integral eigenvalues.

The authors of \cite{ss} introduced fusion association schemes and presented some diagrams to enumerate all association schemes of given small orders according to the partial order defined by fusing.
As shown in the enumeration, the association scheme induced by the icosahedron can be obtained as a fusion of the alternating group $A_4$ of degree $4$ where we identify a finite group $G$ with the association scheme obtained from a regular action of $G$ on itself, but not all eigenvalues of the icosahedron are integral.
On the other hand, every symmetric fusion association scheme of any other non-abelian groups of order 12 has the integral first eigenmatrix.

For more general cases, we introduce the following definition:
an association scheme is said to be \textit{desirable} if the first eigenmatrix of each of its symmetric fusion schemes is integral, otherwise it is said to be \textit{undesirable}.
Our main problem is the following:
\begin{prob}\label{prob:11}
  Characterize desirable association schemes.
\end{prob}


For the remainder of this article we shall write schemes instead of association schemes for short.
It is obvious that every fusion scheme of a desirable scheme is desirable.
Moreover given a desirable scheme $(X,S)$, the subscheme induced by a closed subset and the quotient modulo a closed subset are desirable (see Lemma~\ref{lem:5}), while the direct product (or the other scheme products) of two desirable schemes are not necessarily desirable.
The following are examples of desirable or undesirable association schemes:
\begin{example}\label{ex:11}
  \begin{enumerate}
  \item The scheme of a cyclic group $G$ of order $m$ is undesirable if $m\notin \{1,2,3,4,6\}$ by Corollary~\ref{cor:8}; 
  \item Every symmetric scheme with non-integral first eigenmatrix is undesirable since for any scheme it is also one of its fusion schemes;
  \item Every symmetric scheme with integral first eigenmatrix is desirable by \cite[Lemma~1~(2)]{b};
  \item Every association scheme of rank $2$ is symmetric and integral. This implies that every non-symmetric scheme of rank $3$ is desirable.
  \end{enumerate}
\end{example}
A group $G$ is said to be \textit{desirable} if it induces a desirable scheme by its regular action, otherwise it is said to be \textit{undesirable}.
Then the former example may lead readers to confront the following problem:
\begin{prob} \label{prob:2}
  Find all \textit{desirable} finite groups.
\end{prob}

\begin{remark}
  For a finite group $G$, there is a one-to-one correspondence between the set of fusions of the association scheme induced by $G$ by its regular action and the set of Schur rings over $G$ (see \cite{MP} for the definition of Schur ring).
  Thus Problem~\ref{prob:2} can be stated in terms of Schur rings.
\end{remark}

\begin{remark}
  In connection with Problem~\ref{prob:2} we mention that Bridge and Mena \cite{bm} give a criterion on Cayley graphs over abelian groups with integral eigenvalues which is obtained from a group action.
\end{remark}

By Corollary~\ref{cor:8}, if a finite group $G$ is desirable, then $|G|=2^a3^b$ for some nonnegative integers $a,b$ and the order of each element of $G$ belongs to the set $\{1,2,3,4,6\}$.
But, the converse does not hold because of $A_4$.
In \cite{ado} all \textit{Cayley integral groups} $G$ were classified; the defining property of such a group is that the eigenvalues of any undirected Cayley graph over $G$ are integral. Let $C_n$, $S_n$ and $Q_8$ denote the cyclic group of order $n$, the symmetric group of degree $n$ and the quaternion group, respectively.
It is remarkable that any Cayley integral group is desirable, but the converse does not necessarily hold because $S_3\times C_2$ is desirable but not Cayley integral.
On the other hand, this is a unique desirable group but not Cayley integral, that is exactly the statement of our main result:
\begin{thm}\label{thm:11}
  Every desirable group is isomorphic to one of the following:
  \begin{enumerate}
  \item an abelian group the exponent of which divides $4$ or $6$;
  \item $Q_8\times C_2^m$ for some nonnegative integer $m$;
  \item $S_3$;
  \item $S_3\times C_2$;
  \item $C_3\rtimes C_4=\gn{x,y\mid x^3=y^4=1, y^{-1}xy=x^{-1}}$.
  \end{enumerate}
\end{thm}
In Section~2, we prepare some terminologies on association schemes and groups.
In Section~3, we show a number of desirable groups that will be used in the proof of Theorem~\ref{thm:11}.
In Section~4, we give a proof of our main result, Theorem~\ref{thm:11}.

\section{Preliminaries}
Following \cite{Zie2} we prepare terminologies on association schemes.
Let $X$ be a finite set and $S$ a partition of $X\times X$.
We say that the pair $(X,S)$ is an \textit{association scheme} (or shortly \textit{scheme}) if it satisfies the following:
\begin{enumerate}
\item $1_X:=\{(x,x)\mid x\in X\}$ is an element of $S$;
\item For any $s$ in $S$, $s^*:=\{(y,x) \mid (x,y) \in s\}$ is an element of $S$;
\item For all $s,t,u\in S$ the size of $\{z\in X\mid (x,z)\in s, (z,y)\in t\}$ is constant whenever $(x,y)\in u$.
  The constant is denoted by $a_{stu}$.
\end{enumerate}
For the remainder of this section we assume that $(X,S)$ is an association scheme.
For $s\in S$ we define a matrix $\sigma_s$ over $\mathbb{C}$, which is called the \textit{adjacency matrix} of $s$,
whose rows and columns are indexed by the elements of $X$ as follows:
\[(\sigma_s)_{x,y}=\begin{cases}
    1 & \mbox{if $(x,y)\in s$}\\
    0 & \mbox{if $(x,y)\notin s$.}
  \end{cases}
\]
We shall write $\mathrm{ev}(A)$ as the set of all eigenvalues of a square matrix $A$ over $\mathbb{C}$.
We say that $(X,S)$ is \textit{integral} if $\bigcup_{s\in S}\mathrm{ev}(\sigma_s)\subseteq \mathbb{Z}$.
\begin{remark}
  The first eigenmatrix of $(X,S)$ is defined when $(X,S)$ is commutative, i.e., $\sigma_s\sigma_t=\sigma_t\sigma_s$ for all $s, t\in S$.
  Then the first eigenmatrix of $(X,S)$ is integral if and only if $(X,S)$ is integral.
\end{remark}
We say that $(X,S)$ is
\textit{symmetric} if $\sigma_s$ is symmetric for each $s\in S$, and
\textit{desirable} if $(X,T)$ is integral for each symmetric fusion scheme $(X,T)$ of $(X,S)$.

For a finite group $G$ we set
\[\tilde{G}=\{\tilde{g}\mid g\in G\}\]
where $\tilde{g}=\{(a,b)\in G\times G\mid ag=b\}$.
It is well-known that $(G,\tilde{G})$ is an association scheme (see \cite[Appendix]{Zie1}).
We say that $G$ is \textit{desirable} if $(G,\tilde{G})$ is desirable.

Following \cite{Zie2} we introduce a concept which corresponds to blocks in permutation groups.
For $T\subseteq S$ we say that $T$ is \textit{closed} if
\[\mbox{$\{u\in S\mid a_{stu}>0\}\subseteq T$ for all $s,t\in T$,}\] equivalently,
$\bigcup_{t\in T}t$ is an equivalence relation on $X$ since each digraph $(X,t)$ has a directed cycle because of $|X|<\infty$.
We shall write the equivalence class containing $x$ by $\bigcup_{t\in T}t$ as $xT$.
It is well-known (see \cite[1.5]{Zie1}) that $(xT,\{t\cap (xT\times xT)\mid t\in T\})$
forms an association scheme, which is denoted by $(X,S)_{xT}$, and
the quotient set $X/T$ forms an association scheme, called the \textit{factor scheme} of $(X,S)$ over $T$,
denoted by $(X/T,S\Qd T)$ where
\[\mbox{
    $S\Qd T=\{s^T\mid s\in S\}$ and
    $s^T:=\{(xT,yT)\mid (xT\times yT)\cap s\ne \emptyset\}$.}\]
\begin{lem}\label{lem:5}
  Let $(X,S)$ be a desirable scheme, $x\in X$ and $T$ a closed subset of $S$.
  Then both of $(X,S)_{xT}$ and $(X/T,S\Qd T)$ are desirable.
\end{lem}
\begin{proof}
  Let $(xT,U)$ be a symmetric fusion scheme of $(X,S)_{xT}$ where $U$ is a partition of $xT\times xT$.
  Since each $u\in U$ is a union of elements of the restrictions of $T$ to $xT\times xT$,
  it allows us to fuse elements of $S$ as follows:
  \[\left\{\bigcup_{s\in S\setminus T}s\right\} \cup
    \left\{ \bigcup_{s\in S;s\cap u\ne\emptyset}s
      ~\middle\vert~
      u\in U\right\}\]
  which forms a symmetric fusion scheme of $(X,S)$.
  
  Let $|X/T|=m$.
  Notice that $\bigcup_{s\in S;s\cap u\ne\emptyset}s$ is contained in $\bigcup_{i=1}(x_iT \times x_iT)$ where $x_iT$ with $i=1, \ldots, m$ are the equivalence classes induced by $T$.
  Thus, the characteristic polynomial of $\bigcup_{s \in S; s \cap u \not = \emptyset}s$ is the product of those of $(\bigcup_{s \in S; s \cap u \not = \emptyset}s) \cap (x_iT \times x_iT)$ with $i=1, \ldots, m$ which are mutually equal since $(X,S)_{x_iT}$ have the same structure constants.
  This implies that the characteristic polynomial of $\bigcup_{s \in S; s \cap u \not = \emptyset}s$ is equals the $m$-the power of that of $u$.
  Since $(X,S)$ is desirable, the eigenvalues of the adjacency matrix of $\bigcup_{s\in S; s\cap u\ne\emptyset}s$ are all integers.
  Therefore, $\mathrm{ev}(\sigma_u)\subseteq \mathbb{Z}$ for each $u\in U$.

  Let $(X/T,U)$ be a symmetric fusion of $(X/T,S\Qd T)$, and for $u \in S$ let $A_{TuT}$ denote the adjacency matrix of $\bigcup_{s\in S;s^T\subseteq u^T}s$.
  Notice that $\bigcup_{s \in TuT}s$ is a union of some of the $\{x_iT \times x_jT \mid i,j=1, \ldots, m\}$.
  Moreover, $(x_iT \times x_jT) \subseteq \bigcup_{s \in TuT}s$ if and only if $(x_iT,x_jT) \in u^T$.
  This implies that $A_{TuT}$ is conjugate to
  the Kronecker's product $\sigma_{u^T}\otimes J$
  under the group of permutation matrices
  where $J$ is the all one matrix of degree $m$,
  so that it can be easily checked that
  \[\left\{1_X,\bigcup_{t\in T;t\ne 1_X}t\right\}\cup
    \left\{\bigcup_{s\in S;s^T\subseteq u^T}s
      ~\middle\vert~
      u^T\in U\setminus \{1_{X/T}\}\right\}\]
  is a symmetric fusion of $(X,S)$.

  Since each of the eigenvalues of $A_{TuT}$
  is an integral multiple of an eigenvalue of $\sigma_{u^T}$,
  it follows from the fact that each eigenvalue of $\sigma_{u^T}$ is an algebraic integer
  that $\mathrm{ev}(\sigma_{u^T})\subseteq \mathbb{Z}$ for each $u^T\in U$.
  Therefore, $(X/T,S\Qd T)$ is desirable.
\end{proof}
We frequently use the following without mentioning.
\begin{cor}\label{cor:5}
  Any subgroup or any homomorphic image of a desirable group is desirable.
\end{cor}
\begin{proof}
  Let $G$ be a finite group and $H$ a subgroup of $G$.
  Then $\tilde{H}$ is a closed subset of $\tilde{G}$, and if $H$ is normal in $G$,
  then $(G/H,\tilde{G/H})$ is isomorphic to the factor scheme of $(G,\tilde{G})$ over $\tilde{H}$.
  Applying Lemma~\ref{lem:5} with the homomorphism theorem in group theory we obtain the result.
\end{proof}
\begin{cor}\label{cor:8}
  The order of any element of a desirable group belongs to the set $\{1,2,3,4,6\}$.
  In particular, the order of a desirable group equals $2^a3^b$ for some nonnegative integers $a,b$.
\end{cor}
\begin{proof}
  Let $G$ be a desirable group and $x\in G$ has order $n$.
  By Corollary~\ref{cor:5}, $H:=\gn{x}$ is desirable.
  Since the symmetrization $\{\tilde{y}\cup\tilde{z}\mid y,z\in H;yz=1\}$ forms a symmetric fusion of $(H,\tilde{H})$
  and $\mathrm{ev}(\sigma_{\tilde{y}\cup\tilde{z}})=\{2\cos(2\pi k/n)\mid k\in\mathbb{Z}\}$ with $yz=1$,
  it follows that $n\in \{1,2,3,4,6\}$.
\end{proof}

\section{Undesirable groups of small orders}
By Corollary~\ref{cor:8}, every desirable group has order $2^a3^b$ for some nonnegative integers $a,b$.
But, the converse does not necessarily hold. In this section we collect some undesirable groups of such orders.

\begin{lem}\label{lem:di}
  The dihedral group of order $8$ is undesirable.
\end{lem}
\begin{proof}
  Let $G=\gn{x,y\mid x^4=y^2=1, yxy=x^{-1}}$ be the dihedral group of order $8$.
  Then the following partition of $G$ induces a symmetric fusion of $(G,\tilde{G})$:
  \[\{\{1\},\{x^2\}, \{x,x^3\}, \{y,yx\}, \{yx^2,yx^3\}\}.\]
  Since $\tilde{y}\cup\tilde{yx}$ forms the octagon whose eigenvalues are not all integral,
  the dihedral group of order $8$ is undesirable.
\end{proof}

\begin{lem}\label{lem:al}
  The alternating group of degree $4$ is undesirable.
\end{lem}
\begin{proof}
  The following partition of $A_4$ induces a symmetric fusion of $(A_4, \tilde{A}_4)$:
  \[\{\{1\}, \{(12)(34)\},\{(13)(24), (123),(132), (124),(142)\},\]
  \[\{(14)(23), (134),(143), (234),(243)\}\}.\]
  The third one induces the icosahedron, whose eigenvalues are not necessarily integral.
\end{proof}
\begin{lem}\label{lem:181}
  The semidirect product $(C_3\times C_3)\rtimes C_2$ by the action of the inverse map
  is undesirable.
\end{lem}
\begin{proof}
  Let $\gn{x,y,z\mid x^3=y^3=[x,y]=zxzx=zyzy=z^2=1}$ denote the group given in the statement.
  Then the following partition of $G$ induces a symmetric fusion of $(G,\tilde{G})$:
  \[\{\{1\},\{x,x^2\}, \{y,xy,x^2y, y^2,xy^2,x^2y^2\}, \{z,yz,xy^2z\},\]
  \[ \{xz,x^2z,y^2z,xyz,x^2yz,x^2y^2z\}\}.\]
  Since $\tilde{z}\cup\tilde{yz}\cup\tilde{xy^2z}$ forms a distance-regular graph with intersection array $\{3,2,2,1;1,1,2,3\}$
  whose eigenvalues are not all integral,
  the statement holds.
\end{proof}

\begin{lem}\label{lem:182}
  The direct product $S_3\times C_3$ is undesirable.
\end{lem}
\begin{proof}
  Let $\gn{x,y,z\mid x^3=y^3=[x,y]=[z,x]=zyzy=z^2=1}$ denote the group given in the statement.
  Then the following partition of $G$ induces a symmetric fusion of $(G,\tilde{G})$:
  \[\{\{1\},\{y,y^2\}, \{x,xy,xy^2, x^2,x^2y,x^2y^2\}, \{z,xyz,x^2yz\},\]
  \[ \{yz,y^2z,xz,xy^2z,x^2z,x^2y^2z\}\}.\]
  Since $\tilde{z}\cup\tilde{xyz}\cup\tilde{x^2yz}$ forms a distance-regular graph with intersection array $\{3,2,2,1;1,1,2,3\}$
  whose eigenvalues are not all integral,
  the statement holds.
\end{proof}

\begin{lem}\label{lem:241}
  The direct product $C_2\times C_2\times S_3$ is undesirable.
\end{lem}
\begin{proof}
  Let $G=\gn{x}\times \gn{y}\times \gn{\sigma,\tau}$ denote the group given in the statement
  where $|x|=|y|=|\tau|=2$ and $|\sigma|=3$.
  Then the following partition of $G$ induces a symmetric fusion of $(G,\tilde{G})$:
  \[P\cup xP\mbox{ where}\]
  \[P=\{\{1\},\{xy\}, \{\sigma,\sigma^2\}, \{x\tau,y\sigma\tau \}, \{xy\sigma,xy\sigma^2\},
    \{x\sigma^2\tau,y\sigma^2\tau\},\{x\sigma\tau,y\tau\}\}.\]
  Since $\tilde{x\tau}\cup\tilde{y\sigma\tau}$ forms a disjoint union of two 12-gons
  whose eigenvalues are not all integral, the statement holds.
\end{proof}

\begin{lem}\label{lem:242}
  The semidirect product $(C_3\rtimes C_4)\times C_2$ is undesirable.
\end{lem}
\begin{proof}
  Let $G=H\cup Hy$ where $H=\gn{x}\times \gn{y^2}\times \gn{z}$ is the unique subgroup with index two
  with $|x|=2$, $|y|=4$ and $|z|=3$.
  Then the following partition of $G$ induces a symmetric fusion of $(G,\tilde{G})$:
  \[\left\{\{a\}_{a\in \gn{x,y^2}}, \{az,az^2\}_{a\in \gn{x,y^2}}, \gn{x,y^2}y,T, xT\right\}\]
  where $T:=\{zy,zy^3,xz^2y^3,z^2xy\}$.
  Since the minimal polynomial of the adjacency matrix of the graph induced by $\bigcup_{t\in T}\tilde{t}$ is
  $\lambda(\lambda^2-4)(\lambda^2-16)(\lambda^2-12)$ where $\lambda$ is an indeterminate,
  the statement holds.
\end{proof}

\begin{lem}\label{lem:27}
  There are no non-abelian desirable groups of order 27.
\end{lem}
\begin{proof}
  Notice that there are two non-abelian groups of order 27 those exponents are 9 and 3. Every group with exponent 9 is undesirable by Corollary~\ref{cor:8}.
  Let $G=(\gn{x}\times \gn{y})\rtimes \gn{z}$ where $|x|=|y|=|z|=3$, $zx=xz$ and
  $z^{-1}yz=xy$.
  Then $G$ is a unique non-abelian group of order 27 with exponent 3, and
  the following partition of $G$ induces a symmetric fusion of $(G,\tilde{G})$:
  \[\left\{\{1\},\{x,x^2\}, \{y,y^2,xy,xy^2, x^2y,x^2y^2\},H_1,H_2,H_3\right\}\]
  where $H_1=\{z,z^2,yz,x^2y^2z^2,x^2y^2z,x^2yz^2\}$,
  $H_2=\{xz,x^2z^2,xyz,xy^2z^2,y^2z,xyz^2\}$ and
  $H_3=\{x^2z,xz^2,x^2yz,y^2z^2z, xy^2z,yz\}$.
  Since the minimal polynomial of the adjacency matrix of the graph induced by $\bigcup_{h\in H_1}\tilde{h}$ is
  $\lambda(\lambda+3)(\lambda-6)(\lambda^3-9 \lambda-9)$ where $\lambda$ is an indeterminate,
  the statement holds.
\end{proof}

\begin{prop}\label{prop:10}
  There are no non-abelian desirable groups of order $18$, $24$ or $27$.
\end{prop}
\begin{proof}
  Since the groups as in Lemma~\ref{lem:181} and~\ref{lem:182} are the non-abelian groups of order $18$
  without any element of order $9$, there are no such groups of order $18$.

  The following are the non-abelian groups of order 24 without any element of order 8 or 12:
  \[\mbox{$SL(2,3)$, $ (C_3\rtimes C_4)\times C_2$, $D_{12}\times C_2$, $S_4$, $C_2\times A_4$ and $C_2\times C_2\times S_3$.}\]
  Among them the first, third, fourth, fifth ones are undesirable by Lemma~\ref{lem:al} and Corollary~\ref{cor:5} and the second, last ones are undesirable by Lemma~\ref{lem:241} and \ref{lem:242}. So there are no such groups of order 24.

  By Lemma~\ref{lem:27}, there are no such groups of order 27, the statement holds.
\end{proof}
\section{Proof of our main result}

\begin{lem}\label{lem:13}
  Let $G$ be a desirable group and $a,b\in G$ non-commuting involutions.
  Then $|\gn{a,b}|\in \{6,12\}$.
\end{lem}
\begin{proof}
  Since $\gn{a,b}$ is isomorphic to a dihedral group,
  it follows from Corollary~\ref{cor:8} and Lemma~\ref{lem:di}.
\end{proof}

\begin{lem}\label{lem:30}
  Let $G$ be a desirable group.
  If $a\in G$ normalizes an elementary abelian $2$-subgroup $H$ of $G$,
  then $ab=ba$ for each $b\in H$.
\end{lem}
\begin{proof}
  Suppose the contrary, i.e., $ab\ne ba$ for $a\in N_G(H)$ and $b\in H$
  where $N_G(H)=\{x\in G\mid x^{-1}Hx=H\}$.
  Since $\gn{a^{-i}b a^i\mid i=0,1,\ldots,3}$ is a subgroup of $H$,
  it is an elementary abelian $2$-group of rank at least two.
  By Corollary~\ref{cor:8}, we divide our proof according to $|a|\in \{2,3,4,6\}$.

  If $|a|=2$, then $\gn{a,b}$ is a non-abelian group of order $8$, which is isomorphic to $D_8$, a contradiction to Lemma~\ref{lem:di}.

  If $|a|=4$, then $\gn{a^{-i}b a^i\mid i=0,1,\ldots,|a|-1}=\gn{b,a^{-1}ba}$ since $a^2$ commutes with $b$ by what we proved in the last paragraph.
  If $a^2\in \gn{b,a^{-1}ba}$, then $\gn{a,b}$ is isomorphic to $D_8$ since it has more than one involutions, a contradiction to Lemma~\ref{lem:di}.
  If $a^2\notin \gn{b,a^{-1}ba}$, then $\gn{a^2}$ is central in $\gn{a,b}$ and $|\gn{a,b}/\gn{a^2}|=8$ by $|a^2|=2$.
  If $\gn{a,b}/\gn{a^2}$ is non-abelian, then it is isomorphic to $D_8$, a contradiction to Lemma~\ref{lem:di}.
  If $\gn{a,b}/\gn{a^2}$ is abelian, then $a^{-1}ba=ba^2$, which implies that $bab=a^{-1}$, and hence, $\gn{a,b}$ is isomorphic to $D_8$,
  a contradiction to Lemma~\ref{lem:di}.

  If $|a|=3$ and $\gn{b,a^{-1}ba, aba^{-1}}\simeq C_2^2$,
  then
  $\gn{a,b}=\gn{a, b, a^{-1}b a, aba^{-1}}$ is isomorphic to $A_4$, which contradicts Lemma~\ref{lem:al}.

  If $|a|=3$ and $\gn{b,a^{-1}ba, aba^{-1}}\simeq C_2^3$,
  then $\gn{a,b}$ is a non-abelian group of order $24$, which contradicts Proposition~\ref{prop:10}.

  If $|a|=6$,
  then $a=cd=dc$ for some $d,c\in \gn{a}$ with $|d|=2$ and $|c|=3$.
  Since both $d$ and $c$ centralize $H$, $a$ also centralizes $H$.
\end{proof}

\begin{lem}\label{lem:15}
  If $G$ is a desirable non-abelian $2$-group,
  then $G$ is isomorphic to $Q_8\times C_2^m$ for some nonnegative integer $m$.
\end{lem}
\begin{proof}
  By Lemma~\ref{lem:13}, all involutions of a desirable $2$-group commute for each other.
  This implies that the subgroup, say $K$, generated by all involutions is a normal subgroup,
  which is isomorphic to an elementary abelian $2$-group contained in the center of $G$ by
  Lemma~\ref{lem:30}.

  In order to prove the statement it suffices to show that each cyclic subgroup of $G$ is normal
  by a well-known theorem by Baer and Dedekind (We mimic the same argument as in the proof of \cite[Thm.~2.13]{ado}).
  Suppose the contrary, i.e., $a^{-1}ba \notin \gn{b}$ for $a,b\in G\setminus K$, namely, $|a|=|b|=4$.
  Let $L$ denote the subgroup of $\gn{a,b}$ generated by the involutions of $\gn{a,b}$.
  Since $L$ is central in $G$ and $b^2\in L$,
  $b\gn{b^2}\in \gn{a,b}/\gn{b^2}$ is an element of order two.
  Since $\gn{a,b}/\gn{b^2}$ is a desirable $2$-group by Corollary~\ref{cor:5},
  it follows from the same argument as in the last paragraph that $b\gn{b^2}$ is contained in the center of $\gn{a,b}/\gn{b^2}$.
  Therefore, $a^{-1} ba \in b\gn{b^2}\subseteq \gn{b}$, a contradiction.
\end{proof}

\begin{lem}\label{lem:155}
  If $G$ is a desirable $3$-group,
  then $G$ is isomorphic to $C_3^m$ for some nonnegative integer $m$.
\end{lem}
\begin{proof}
  Suppose that $G$ is a desirable non-abelian $3$-group of the least order.
  Let $x,y\in G$ with $xy\ne yx$.
  By Corollary~\ref{cor:5}, $\gn{x,y}$ is a desirable non-abelian $3$-group, which is of exponent three by Corollary~\ref{cor:8}.
  By the minimality of $|G|$ we have $G=\gn{x,y}$.
  Since $G$ has a non-trivial center, there exists a non-identity element $z\in Z(G)$.
  By the minimality of $|G|$ and Corollary~\ref{cor:5}, $G/\gn{z}$ is an elementary abelian $3$-group of rank two.
  This implies that $|G|=|\gn{x,y}|=|G/\gn{z}||\gn{z}|=27$, which contradicts Proposition~\ref{prop:10}.

  Next we suppose that $G$ is a desirable abelian $3$-group.
  Since there are no element of order lager than $3$ in $G$ by Corollary~\ref{cor:8}, $G$ is isomorphic to $C_3^m$ for some nonnegative integer $m$.
\end{proof}

\begin{lem}\label{lem:33}
  Let $G$ be a desirable group and $a\in G$.
  If $a$ normalizes an elementary abelian $3$-group $H$ of $G$,
  then $a^{-1}ba\in \{b,b^{-1}\}$ for each $b\in H$.
\end{lem}
\begin{proof}
  Suppose the contrary, i.e., $a^{-1}ba\notin \{b,b^{-1}\}$ for $a\in N_G(H)$ and $b\in H$.
  Since $\gn{a^{-i}b a^i\mid i=0,1,\ldots,|a|-1}$ is a subgroup of $H$,
  it is an elementary abelian $3$-group of rank at least two.
  By Corollary~\ref{cor:8}, we divide our proof according to $|a|\in \{2,3,4,6\}$.

  If $|a|=2$, then $\gn{a,b}$ has order $18$, which contradicts Corollary~\ref{cor:5} and Proposition~\ref{prop:10}.

  If $|a|=4$, then
  \[\gn{b,a^{-1}ba, a^2ba^2, aba^{-1}}\]
  is an elementary abelian $3$-group contained in $H$.
  Note that $\gn{b, a^2ba^2}$ is an elementary abelian $3$-group of rank at most two.
  If $a^2ba^2\notin \{b,b^{-1}\}$, then $\gn{a^2, b}$ is a non-abelian group of order $18$, which contradicts Proposition~\ref{prop:10}.
  If $a^2 ba^2=b$,
  then $a^2$ is central in $\gn{a,b}$, and $a^{-1}ba\notin b\gn{a^2}$ by the assumption that $a^{-1}ba\in H$ and $a^2\notin H$.
  This implies that $\gn{a,b}/\gn{a^2}$ is a non-abelian group of order 18,
  which contradicts Proposition~\ref{prop:10}.
  If $a^2ba^2=b^{-1}$, then $\gn{a^2, b, a^{-1}ba}$ is a non-abelian group of order $18$,
  which contradicts Proposition~\ref{prop:10}.

  If $|a|=3$, then $\gn{a,H}$ is abelian
  by Lemma~\ref{lem:155}, and hence $a$ centralizes $H$.

  If $|a|=6$,
  then $a=cd=dc$ for some $d,c\in \gn{a}$ with $|d|=2$ and $|c|=3$.
  Since both $d$ and $c$ normalize $\gn{b}$, $a$ also normalizes $\gn{b}$.
\end{proof}

It is well-known that a minimal normal subgroup of a finite group is the direct product of isomorphic simple groups
(see \cite{DM}). Applying this fact with Corollary~\ref{cor:8} we obtain that any minimal normal subgroup of a desirable group is
isomorphic to an elementary abelian $p$-group for some $p\in \{2,3\}$.
\begin{lem}\label{lem:50}
  Let $G$ be a desirable group and $N$ a minimal normal subgroup of $G$.
  If $N\simeq C_2^m$ and there exist $a,b\in G$ with $|a|=|b|=2$ and $|ab|=3$,
  then $m= 1$.
\end{lem}
\begin{proof}
  We claim that $N\cap \gn{a,b}=1$, otherwise $\gn{a,b}$ contains an involution in $N$. Since all involutions of $\gn{a,b}$ are conjugate in $\gn{a,b}$,
  it follows that $\gn{a,b}\subseteq N$, a contradiction to $|ab|=3$ and $N\simeq C_2^m$.
  By Lemma~\ref{lem:30}, $N$ is central, and hence, if $N$ has a subgroup $H$ of order $4$, then, by Lemma~\ref{lem:30},
  $|H\gn{a,b}|=24$, which contradicts Proposition~\ref{prop:10}.
  Therefore, $m= 1$.
\end{proof}

\begin{lem}\label{lem:60}
  Let $G$ be a desirable group and $N$ a minimal normal subgroup of $G$.
  If $N\simeq C_3^m$ and $a,b\in G$ with $|a|=|b|=2$ and $|ab|=3$,
  then $N=\gn{ab}$.
\end{lem}
\begin{proof}
  Suppose $c\in N\setminus \gn{ab}$.
  By Lemma~\ref{lem:33}, $\gn{a,b}$ normalizes $\gn{c}$, and hence,
  $|\gn{a,b,c}|=18$, which contradicts Proposition~\ref{prop:10}.
\end{proof}

\begin{lem}\label{lem:70}
  Let $G$ be a desirable group.
  If $G$ has two non-commuting involutions, then
  $|G|\in \{6,12\}$.
\end{lem}
\begin{proof}
  Use induction on $|G|$. Suppose that $G$ is a desirable group with two non-commuting involutions and
  $|G|$ is minimal such that $|G|\notin \{6,12\}$.
  Note that any two non-commuting involutions generate the dihedral group of degree $3$ or $6$ by Lemma~\ref{lem:13},
  and each of the cases contains two non-commuting involutions whose product has order three.
  Let $N$ be a minimal normal subgroup of $G$ and $a,b\in G$ such that $|a|=|b|=2$ and $|ab|=3$.
  Recall that $N$ is an elementary abelian $p$-group for some $p\in \{2,3\}$.
  Applying Lemma~\ref{lem:50} and \ref{lem:60} we obtain that
  $N\simeq C_2$ and $N\cap \gn{a,b}=1$, or $N=\gn{ab}$.

  If $N\simeq C_ 2$ and $N\cap \gn{a,b}=1$, then $G/N$ is a desirable group with two non-commuting involutions.
  By the minimality of $|G|$, $|G/N|\in \{6,12\}$. Since $|G|\ne 12$ and $|N|=2$, it follows that $|G/N|=12$, and hence $G$ is a non-abelian group of order $24$,
  a contradiction to Proposition~\ref{prop:10}.

  Suppose $N=\gn{ab}$.

  We claim that $8\nmid |G/N|$.
  Otherwise there exists a subgroup $L/N$ of $G/N$ such that $|L/N|=8$ and $a,b\in L$ by Sylow's theorem, implying that
  $L$ is a non-abelian group of order $24$, a contradiction to Proposition~\ref{prop:10}.

  If $G/N$ has two non-commuting involutions, then $|G/N|\in \{6,12\}$ by the minimality of $|G|$.
  By Proposition~\ref{prop:10}, $|G/N|=12$. Since $G/N$ is not isomorphic to $A_4$ by Lemma~\ref{lem:al},
  $G/N$ has a minimal normal subgroup $N_1/N$ of order $3$ by the classification of groups of order $12$.
  Since $|N_1|=9$, $\gn{a,b,N_1}$ is a non-abelian group of order $18$ by Lemma~\ref{lem:33}, a contradiction to Proposition~\ref{prop:10}.
  Therefore, we conclude that all involutions of $G/N$ commute for each other, and the subgroup of $G/N$ generated by all involutions
  is a normal subgroup of $G/N$ which is an elementary abelian $2$-group.

  By the last claim, it suffices to show that $3\nmid |G/N|$.
  Otherwise, there exists $L/N\leq G/N$ such that $N\leq L$ and $|L/N|=3$.
  Since $aN$ is an involution of $G/N$, it is central by Lemma~\ref{lem:30}.
  Thus, $\gn{aN,L/N}$ is a subgroup of order $6$.
  This implies that $\gn{L,a,b}$ is a non-abelian group of order $18$, a contradiction to Proposition~\ref{prop:10}.
\end{proof}

\begin{lem}\label{lem:80}
  Let $G$ be a non-abelian desirable group any two involutions of which are commute.
  Then the subgroup of $G$ generated by all involutions is normal in $G$, and
  $G$ is isomorphic to $C_3\rtimes C_4$ unless $G$ is a $2$-group.
\end{lem}
\begin{proof}
  The first statement is obvious.
  Let $K$ be the subgroup of $G$ generated by all involutions of $G$.
  Since $G$ is non-abelian, $G$ is not a $3$-group by Lemma~\ref{lem:155}.
  This implies that $K$ has a subgroup $L$ of order two.
  By Lemma~\ref{lem:30}, $L$ is central in $G$.

  We use the induction on $|G|$ to prove the second statement.
  Let $G$ be a non-abelian desirable group of the least order such that
  all involutions of $G$ commute for each other and $G$ is neither $2$-group nor $G\simeq C_3\rtimes C_4$.

  If $G/L$ has two non-commuting involutions, then $|G/L|\in \{6,12\} $ by Lemma~\ref{lem:70}.
  By Proposition~\ref{prop:10}, $|G|=12$, and hence $G\simeq C_3\rtimes C_4$ by the classification of groups of order $12$, a contradiction.

  If $G/L$ has no two non-commuting involutions and non-abelian then, by the minimality of $|G|$,
  $G/L\simeq C_3\rtimes C_4$ or a $2$-group. But, the former case does not occur by Proposition~\ref{prop:10},
  and the latter case implies that $G$ is a $2$-group, a contradiction.

  Suppose that $G/L$ is an abelian group any two involutions of which are commute.
  We claim that $a\in Z(G)$ for each element $a\in G $ with $|a|=3$.
  Otherwise, $ab\ne ba $ for some $b\in G$.
  Since $G/L$ is abelian, $b^{-1}a b=al$ for a non-identity $l\in L$.
  Since $l\in Z(G)$ and $|l|=2$, it follows that $|al|=6$, which contradicts $|a|=3$.
  Applying the claim for an element $c\in G $ of order 6 we obtain from $c=c^4c^3$, $|c^4|=3$ and $|c^3|=2$ that
  each element of order $2$, $3$ or $6$ is in the center of $G$.
  This implies that there exist $a,b\in G$ such that $|a|=|b|=4$ and $ab\ne ba$ since $G$ is non-abelian.
  Applying Lemma~\ref{lem:15} we conclude that $G$ has a unique Sylow $2$-subgroup, which has a subgroup isomorphic to $Q_8$.
  Since $G$ is not a $2$-group by the assumption, it follows from the claim that there exists a subgroup of $G$ isomorphic to $C_3\times Q_8$,
  a contradiction to Proposition~\ref{prop:10}.
\end{proof}

\begin{flushleft}
  \begin{large}
    Proof of Theorem~\ref{thm:11}
  \end{large}
\end{flushleft}
\begin{proof}
  Suppose that $G$ is a non-abelian desirable group.
  If $G$ has two non-commuting involutions, then $|G|\in \{6,12\}$ by Lemma~\ref{lem:70}.
  If $G$ has no two non-commuting involutions, then
  $G\simeq C_3\rtimes C_4$ or a $2$-group by Lemma~\ref{lem:80}, which is eliminated by Lemma~\ref{lem:15}.
  Since all desirable groups of order 6 or 12 are known to be $S_3$, $S_3\times C_2$ or $C_3\rtimes C_4$,
  this completes the proof.
\end{proof}

\bibliography{ref}
\bibliographystyle{abbrv}

\end{document}